\documentclass[11p,reqno]{amsart}
\usepackage{amsmath}
\usepackage{amsfonts}
\usepackage{amssymb}
\usepackage{graphicx}
\usepackage{color}
\newtheorem{theorem}{Theorem}[section]
\newtheorem*{theorem*}{Theorem}
\newtheorem{lemma}{Lemma}[section]
\newtheorem{corollary}[theorem]{Corollary}
\newtheorem{proposition}{Proposition}[section]

\def\Ric{\text{Ric}}

\def\Ric{\operatorname{Ric}}

\def\Ric{{\operatorname{Ric}}}

\numberwithin{equation}{section}

\begin{document}
	\title[Multiplicity Estimates ]{A multiplicity estimate for the Jacobi operator of a nonflat Yang-Mills field over $\mathbb{S}^m$}

\author{Lei Ni}
\address{Lei Ni. Department of Mathematics, University of California, San Diego, La Jolla, CA 92093, USA}
\email{leni@ucsd.edu}


\subjclass[2010]{}

\begin{abstract}  Here we provide refinements of the stability results of Simons and Xin, concerning the stability of Yang-Mills fields and harmonic maps respectively. The result also implies the earlier Morse index estimates for both cases.
\end{abstract}

\maketitle

\section{Introduction}

The stability is a central issue in variational problems in analysis and geometry. It was proven that there is no nontrivial stable Yang-Mills fields on $\mathbb{S}^m$ for $m\ge 5$ \cite{Simons, BLS, BL} and there is no nonconstant stable harmonic maps from $\mathbb{S}^m$ for $m\ge 3$ \cite{Xin}. The stability of minimal surfaces was studied also extensively \cite{Simons-Ann}.

Later the above results were strengthened by effective estimates on the lower bound of the Morse index. Before we state the results let's first recall some notations and  definitions.

Let $u: (M^m, g)\to (N^n, h)$ be a smooth map between two Riemannian manifolds. Define
$$
\mathcal{E}(u)=\frac{1}{2}\int_M |du|^2\, d\mu_g, \mbox{ where } |du|^2=\sum_{i=1}^m |du(e_i)|^2.
$$
Here  $\{e_i\}$ is  an orthonormal basis of  $T_xM$ for any $x\in M$. The Einstein convention is applied below. If $u_s=u(s,\cdot): (-\epsilon, \epsilon)\times M\to N$ is a family of maps (a variation), we can consider the first and second variations of $\mathcal{E}(u_s)$. The critical point is called a {\it harmonic map}. The second variation formula (\cite{ES}, \cite{Xin} ) gives that, at a harmonic map $u=u_0$,
\begin{equation}\label{eq:11}
\left.\frac{d^2}{ds^2}
\mathcal{E}(u_s)\right|_{s=0}=\int_M\langle \nabla^*\nabla V- R^N_{V, du(e_i)}du(e_i), V\rangle\, d\mu_g
\end{equation}
Here $V$ is the variational vector field along $u: M \to N$, namely $\left. V=du_s(\frac{\partial}{\partial s}) \right|_{s=0}$, which can be viewed as a section of bundle $ E=u^{-1} (TN)$ over $M$, and  $\nabla^*\nabla V=-\nabla^2_{e_i, e_i} V $ is  $-1$ times the trace of the Hessian operator of the bundle $E$. The convention of the curvature is that $\langle R_{X, Y}Y, X\rangle>0$ for the standard sphere $\mathbb{S}^n$. The second order self-adjoint elliptic  operator $\mathcal{J}_u=\nabla^*\nabla -R^N_{(\cdot),  du(e_i)}du(e_i)$ is called the Jacobi operator. The harmonic map is called stable if all the eigenvalues of $\mathcal{J}_u$ is nonnegative. The constant map clearly is the minimizer of $\mathcal{E}(u)$. The stable maps are the local minimizers. The total number of negative eigenvalues (multiplicity counted) of $\mathcal{J}_u$ is called the Morse index of $u$ (denoted as $\iota(u)$). The dimension of null space $\mathcal{N}=\{V\,|\, \mathcal{J}_u (V)=0\}$ is called the nullity of $u$. The following result of \cite{ElSou} improves Xin's theorem.

\begin{theorem}\label{thm:ElSou} For $m\ge 3$, and any Riemannian manifold $(N, h)$, let $u: \mathbb{S}^m\to (N, h)$ be a nonconstant harmonic map. Then the Morse index of $u$, $\iota(u)\ge m+1$.
\end{theorem}

Xin's theorem was preceded/motivated by a corresponding important result of Simons for Yang-Mills fields and Yang-Mills connections \cite{Simons, BLS}, which is defined to be the critical points of the Yang-Mills functional $\mathcal{YM}(D)$ for connections $D$. Recall that for a connection $D$ on a principle $G$-bundle and associated $G$-vector bundle $E$ over a Riemannian manifold $(M, g)$:
$$
\mathcal{YM}(D)=\frac{1}{2}\int_M \|R^D\|^2\, d\mu_g.
$$
Here $R^D$ is the curvature of the connection $D$ on $E$. The norm is taken with respect to the Riemannian metric of $M$ and  the $Ad(G)$-invariant metric on the Lie algebra $\mathfrak{g}$ of $G$ (usually a subalgebra of $\mathfrak{so}(n)$ where $n=\dim(E)$). The first and second variations of $\mathcal{YM}(D_s)$ can be defined and calculated similarly for a family of connections $D_s$. In particular, a critical point is called a Yang-Mills connection with its curvature $R^D$ being called a Yang-Mills field. The second variational formula (cf. Theorem 2.21 of \cite{BL}) at a Yang-Mills connection is given by
\begin{equation}\label{eq:12}
\left.\frac{d^2}{ds^2}
\mathcal{YM}(D_s)\right|_{s=0}=\int_M\langle (d_D)^*d_D B+\mathcal{R}^D(B), B\rangle\, d\mu_g
\end{equation}
where $B=\left.\frac{d}{ds} D_s\right|_{s=0} \in \Omega^1(\mathfrak{g}_E)$ is the variational $1$-form of the connections. The operator $d_D$ is the exterior derivative on $\Omega^*(\mathfrak{g})$ and $(d_D)^*)$ is its conjugate. One can consult \cite{BL} for details of the notations. The minimizers include the self-dual (anti-self dual) ones in dimensional four and the flat connections.  The associated second order self adjoint operator $\mathcal{J}_D= (d_D)^*d_D +\mathcal{R}^D $ is the corresponding Jacobi operator. One defines the Morse index and nullity of $D$ similarly. The following result of \cite{NU} extends Simons' theorem.

\begin{theorem}\label{thm:NU} For any nonflat Yang-Mills connection $D$ on any vector bundle $E$ over the $m$-sphere $\mathbb{S}^m$, $m\ge 5$ with the standard metric, the Morse index $\iota(D)\ge m+1$.
\end{theorem}

The goal of this note is to prove  the following refinement of Theorems \ref{thm:ElSou} and \ref{thm:NU}.

\begin{theorem}\label{thm:main} (i) For $m\ge 3$ and any Riemannian manifold $(N, h)$, and any nonconstant harmonic map $u: \mathbb{S}^m\to (N, h)$, let $\mathcal{J}_u$ be the Jacobi operator. Let $E^u_{\lambda}:=\{X\, |\mathcal{J}_u(X)=\lambda X\}$ be the space of the eigenvector fields with eigenvalue $\lambda$. Then
$\dim(E^u_{-(m-2)})\ge m+1$. In particular, the smallest eigenvalue has the estimate $\lambda_1(\mathcal{J}_u)\le -(m-2)$.

(ii) For any nonflat Yang-Mills connection $D$ on any vector bundle $E$ over the $m$-sphere $\mathbb{S}^m$, $m\ge 5$ with the standard metric, let $\mathcal{J}_D$ be the Jacobi operator. Let $E^D_{\lambda}:=\{B\, |\mathcal{J}_D (B)=\lambda B\}$ be the space of the eigenforms with eigenvalue $\lambda$. Then
$\dim(E^D_{-(m-4)})\ge m+1$. In particular, the smallest eigenvalue has the estimate $\lambda_1(\mathcal{J}_D)\le -(m-4)$.
\end{theorem}

Since the multiplicity of negative eigenvalues $-(m-2)$, or $-(m-4)$, in each case above, is given by the dimension of the eigenspace, which contributes to the Morse index, the theorem above does imply the earlier results of Simons \cite{Simons}, Xin \cite{Xin}, El Soufi \cite{ElSou},  Nayatani and Urakawa \cite{NU}. The proof of the part (i) is known to experts even though we did not find the statement of the result in the literature. Our contribution is  a very simple proof in this case. A result for minimal submanifolds in $\mathbb{S}^n$ is also obtained. In particular, Theorem \ref{thm:main2} generalizes Simons' eigenvalue estimate for the Jacobi operator for the minimal hypersurface in $\mathbb{S}^{n}$ to high codimensional case.  We remark that besides the stability and Morse index estimate, the nullity estimates for non totally geodesic closed  minimal spheres/varieties in higher dimensional spheres were obtained in \cite{Simons-Ann} much earlier. The stability issue for Yang-Mills fields was also studied in dimension four in \cite{Taubes}. It is  interesting to study for what manifold $N$ ($M$) and the property of  the map $u$ (of Yang-Mills connection $D$)  the equality cases in the above theorem for  either harmonic maps and for the Yang-Mills fields hold. The Morse index estimate for harmonic 2-sphere plays an important role in the application \cite{MM}, where a lower estimate was proved for harmonic $2$-sphere into a manifold with positive isotropic curvature.  The proof of \cite{MM} used a different approach which relied on some curvature conditions of $N$ and a decomposition  theorem of holomorphic vector bundles over $\mathbb{S}^2$. The way of using   the complex structure of the normal bundle and the construction of holomorphic sections  have their precedences in \cite{Ejiri} (cf. also Theorem 3.1.5 of \cite{Simons-Ann} and \cite{SY}). There exist more recent lower Morse index estimates \cite{Kap, KNPP} for two-sphere and projective planes inside a high dimensional spheres, related to the extremal metrics for higher eigenvalues, in terms of the so-called spectral index. Another interesting question is to have a lower estimate on the multiplicity and Morse index for the maps or minimal submanifolds which are not holomorphic or anti-holomorphic between K\"ahler manifolds when one of them is an irreducible Hermitian symmetric space.

\section{Preliminaries}
Let $E$ be a Riemannian vector bundle over $M$ and let $D$ be a connection compatible with the metric. We shall denote the Riemannian curvature of $(M, g)$ by $R^M$. The curvature of $D$ shall be denoted as $R^D$. Recall that
$R^D_{X, Y}=D_XD_Y-D_YD_X -D_{[X, Y]}$ is valued in the endomorphism bundle of $E$.  In fact the image is in $\mathfrak{so}(n)$ with $n$ being the dimension of $E$.  Recall that (cf. (2.3) of \cite{Ni-holo})  for any $A\in \mathfrak{so}(n)$,
\begin{equation}\label{eq:21}
\langle A, z\wedge w\rangle =\langle A(w), z\rangle.
\end{equation}
In the case of the Riemannian curvature we have $R_{X, Y}=R(X\wedge Y)$ and with our convention
$$
\langle R(X\wedge Y), Z\wedge W\rangle =\langle R_{X, Y} W, Z\rangle=R(X, Y, Z, W).
$$
 Let $\Omega^p(E)$ be the $p$-forms valued in $E$. The following Bochner formula is well known (cf. Proposition 1.3.4 of \cite{Xin-book}).
\begin{proposition}\label{prop:21}Let $\omega\in \Omega^p(E)$. Then
$$
\Delta_{d_D}\omega \doteqdot \left(d_D d_D^*+ d_D^* d_D \right) \omega=-\Delta \omega +S
$$
where  $\Delta =\sum_{j=1}^n \nabla^2_{e_j, e_j}$, $S$ is defined for any $X_1, \cdots, X_p\in \mathcal{X}(M)$ that
$$
S(X_1, \cdots, X_p)=-\sum_{j=1}^m\sum_{k=1}^p (R_{e_j, X_k} \cdot \omega)( X_1, \cdots, \hat{(e_j)}_k, \cdots, X_p).
$$
Here $\{e_j\}$ is an orthonormal basis of $T_pM$.
\end{proposition}
The curvature term  $R_{e_\alpha, X_k}\cdot \omega$ acts on $\wedge^p T^*M\otimes E$ as a derivation.
The result follows from the corresponding formulae for $d_D$ and $d_D^*$.
 \begin{lemma} For $\omega\in \Omega^k(E)$ and vector fields $X_0, \cdots, X_k$,
\begin{eqnarray}\label{eq:exter-pform-2}
d_D\omega(X_0, \cdots, X_k)&=&\sum_{i=0}^k (-1)^i D_{X_i}(\omega(X_0, \cdots, \hat{X_i}, \cdots, X_k))\\
&\quad&+\sum_{0\le i<j\le k} (-1)^{i+j} \omega([X_i, X_j], X_0, \cdots, \hat{X_i}, \cdots, \hat{X_j}, \cdots, X_k).\nonumber
\end{eqnarray}
\end{lemma}
\begin{lemma} Let $\nabla$ be a torsion free connection on $TM$. For $\omega\in \Omega^k(E)$,
\begin{equation}\label{eq:exter-covar1-2}
d_D\omega(X_0, X_1, \cdots, X_k)=\sum_{i=0}^k (-1)^i (D_{X_i} \omega)(X_0, \cdots, \hat{X}_i, \cdots, X_k).
\end{equation}
\end{lemma}
Here  $(D_{Y} \omega)(X_1, \cdots,  X_k)=D_{Y}(\omega(X_1, \cdots, X_k))-\sum_{i=1}^k \omega(X_1, \cdots, \nabla_Y X_i, \cdots X_k).$ If $\{e_s\}$ is any frame with $(g_{st})=(\langle e_s, e_t\rangle)$,
$$
(d^*_D\omega)(X_1, \cdots, X_{k-1})=-\sum_{s, t} g^{st}(D_{e_s} \omega)(e_t, X_1, \cdots, X_{k-1}).
$$
We also apply the formula to the case $E=\mathfrak{so}(E')$ with $E'$ being a Riemaniann vector bundle. The curvature of $E'$ acts on $\mathfrak{so}(E')$ via the adjoint action in this case.

Let $u:(M, g) \to (N, h)$ be a harmonic map. Let $E=u^{-1} (TN)$ be the pull back bundle equipped with the Riemannian metric $h$. The curvature of $E$ will be the pull back of $R^N$, namely $u^* R^N$. Now $du: TM\to TN$ can be viewed a $1$-form valued in $E$. Namely $du\in \Omega^1(E)$. In fact for any smooth map $d_D du =0$. The harmonic equation amounts to $d^*_D du=0$.  In this case $du$ is $d_D$-harmonic with $D$ being the associated connection on $E$. Hence  one  can apply Proposition \ref{prop:21} and  expresses the curvature term more explicitly to arrive the well-known formula for a harmonic map $u$:
\begin{equation}\label{eq:24}
\Delta du= - \sum_{\alpha}R^N_{ du(\cdot), du(e_\alpha)} du(e_\alpha)+du(\operatorname{Ric}^M(\cdot)).
\end{equation}

For the case of Yang-Mills fields concerning a Riemannian vector bundle $E$ over $(M, g)$, we apply the formula to the $p$-forms valued in $\mathfrak{so}(E)$. In particular for the $1$-form $B$ and $2$-form $\varphi$  the following results (cf. Theorem 3.2 and Theorem 3.10 of \cite{BL}) follow from Proposition \ref{prop:21}.

\begin{proposition}\label{prop:22}  Let $B\in \Omega^1(\mathfrak{so}(E))$. Then
\begin{equation}\label{eq:25}
\Delta_{d_D}B=-\Delta B +B(\Ric^M(\cdot ))+\mathcal{R}^D(B)
\end{equation}
with $\mathcal{R}^D(B)(X)=\sum_{j=1}^m [R^D_{e_j, X}, B(e_j)].$  Let  $\varphi\in \Omega^2(\mathfrak{so}(E))$. Then
\begin{equation}\label{eq:26}
\Delta_{d_D} \varphi=-\Delta \varphi +\varphi \cdot (2 \operatorname{Ric}^M\wedge \operatorname{id}-2R^M)+\mathcal{R}^D (\varphi).
\end{equation}
Here $\mathcal{R}^D(\varphi)(X, Y)=\sum_{j=1}^m \left( [R^D_{e_j, X}, \varphi_{e_j, Y}]-[R^D_{e_j, Y}, \varphi_{e_j, X}]\right)$.
\end{proposition}
Here we use the convention as in \cite{Ni-holo} about $A\wedge B$, namely $$A\wedge B(x\wedge y)=\frac{1}{2}\left(A(x)\wedge B(y)+B(x)\wedge A(y)\right).$$ We also view the Riemannian curvature  as an operator $R^M: \mathfrak{so}(T_pM)\to \mathfrak{so}(T_pM)$ defined as above, namely $\langle R^M(x\wedge y), z\wedge w\rangle =R^M(x,y, z, w)$.

\section{Proof of Theorem \ref{thm:main}}

 We first derive some useful,  also well-known, formulae for the linear functions of $\mathbb{R}^{m+1}$ restricted to the unit sphere $\mathbb{S}^m$. We use $D$ to denote the standard derivative/connection on $\mathbb{R}^{m+1}$ and $\nabla$, the derivative/Levi-Civita connection of $\mathbb{S}^m$. For linear function $\ell(x)$, it is well-known that $\langle D \ell, x\rangle =\ell$. Hence $\nabla \ell =D\ell -\ell \cdot x$. Since $D^2\ell =0$, it is easy to compute that for $X, Y$ tangent to $\mathbb{S}^m$ we have
\begin{eqnarray*}
(\nabla^2 \ell)(X, Y)&=&XY \ell -\langle \nabla \ell, \nabla_X Y\rangle
= XY\ell -\langle D  \ell, \nabla_X Y\rangle \\
&=&XY\ell -\langle D  \ell, D_X Y\rangle+\langle D\ell, B(X, Y)\rangle=\langle \ell \cdot x, B(X, Y)\rangle=-\ell \langle X, Y\rangle.
\end{eqnarray*}
Here we have followed the convention of \cite{Simons-Ann} to define the second fundamental form $B(X, Y):=D_X Y-\nabla_X Y$, and  used that the second fundamental form of the sphere $B(X,Y)=-\langle X, Y\rangle \cdot x $. This then  implies that
\begin{equation}\label{eq:31}
\nabla_X \nabla \ell=-\ell X; \quad \quad \Delta \nabla \ell =- \nabla \ell.
\end{equation}
The first one is obvious. The second one can be done via the first and the commutator formula (using $\operatorname{Ric}^{\mathbb{S}^m}=(m-1)\operatorname{id}$). Or we choose a normal frame $\{e_\alpha\}$ with $\nabla_{e_\beta} e_\gamma=0$ at the point concerned, and compute
\begin{eqnarray}
\Delta \nabla \ell&=& \nabla_{e_\alpha}\nabla_{e_\alpha} \nabla \ell =-\nabla_{e_\alpha} \left(\ell e_\alpha\right)\nonumber\\
&=& -\langle e_\alpha,  \nabla \ell\rangle e_\alpha=-\nabla \ell. \label{eq:32}
\end{eqnarray}

Now we prove the part of Theorem \ref{thm:main} concerning  the harmonic maps. It follows from the two propositions below .

\begin{proposition}\label{prop:31} Assume that $u: \mathbb{S}^m \to (N, h)$ is a harmonic map. The associated section $X_\ell =du(\nabla \ell)$ of $E$ satisfies
$$
\mathcal{J}_u (X_\ell) =-(m-2) X_\ell.
$$
Namely if $X_\ell\ne 0$, it is an eigenvector of the Jacobi operator with eigenvalue $-(m-2)$.
\end{proposition}
\begin{proof} Direct calculations with the help of (\ref{eq:24}) yield that
\begin{eqnarray*}
\Delta (du (\nabla \ell))&=&(\Delta du)(\nabla \ell)+2(\nabla_{e_j} du)(\nabla_{e_j} \nabla \ell)+du(\Delta \nabla \ell)\\
&=&-R^N_{du(\nabla\ell), du(e_j)} du(e_j)+du(\operatorname{Ric}^M (\nabla \ell))+2(\nabla_{e_j} du)(\nabla_{e_j} \nabla \ell)+du(\Delta \nabla \ell)\\
&=& -R^N_{ du(\nabla\ell), du(e_j),} du(e_j)+(m-1)du(\nabla \ell)-2\ell (\nabla_{e_j} du)(e_j)-d u(\nabla \ell)\\
&=&  -R^N_{du(\nabla\ell), du(e_j) } du(e_j)+(m-2)du(\nabla \ell).
\end{eqnarray*}
In the above, from line 2 to line 3 we used (\ref{eq:31}), (\ref{eq:32}), and that $\Ric^{\mathbb{S}^m}=(m-1)\operatorname{id}$. From line 3 to line 4 we used the harmonic map equation $(d_D)^* du=- (\nabla_{e_j} du)(e_j)=0$.
\end{proof}

\begin{proposition}\label{prop:32} If for a smooth map $u$ and for  some linear function $\ell$, $X_\ell=du(\nabla \ell)=0$, $u$ must be a constant map.
\end{proposition}
\begin{proof} This was proved in \cite{ElSou} for harmonic maps. We provide an argument below for the above general result for any $C^1$-maps. It is easy to see that $\ell$ attains a unique maximum point and a unique minimum point on $\mathbb{S}^m$. Consider the flow $\Psi_s$ generated by $\nabla \ell$. It has two fixed points.  One of them is a source $p_{-\infty}$ and the other is a sink $p_\infty$ due to the explicit Hessian of $\ell$ provided in (\ref{eq:31}). Consider the image  curve  $u(\Psi_s(p))$. Since it stays in a compact region,
$\lim_{s\to \infty} u(\Psi_s(p))=u(\lim_{s\to \infty}\Psi_s(p))=u(p_\infty)$, if at $s=0, \Psi_s(p)=p\ne p_{-\infty}$. On the other hand
$$
\frac{d}{ds}\left(u(\Psi_s(p)\right)=du\left(\frac{d}{ds} \Psi_s(p)\right)=du\left(\left.\nabla \ell\right|_{\Psi_s(p)}\right)=0.
$$
This implies that $u(p)=u(p_\infty)$ for any $p\in \mathbb{S}^{m}$ for all $p \ne p_{-\infty}$, which proves the claim.
\end{proof}

Since the space $\mathcal{H}^1$ of all linear functions of $\mathbb{R}^{m+1}$ is of dimension $m+1$, the gradient of their restrictions on $\mathbb{S}^m$ is of the same dimension due to the homogeneity. The above proposition asserts that $\{X_\ell=du(\nabla \ell)\,|\ell\in \mathcal{H}^1 \}$ is also a $(m+1)$-dimensional linear space if $u$ is not a constant map. This proves the harmonic map part of Theorem \ref{thm:main}.

Now we prove the part of  Theorem \ref{thm:main} concerning  Yang-Mills fields. We need the following lemma (cf. Lemma 7.3 of \cite{BL}), which can also be obtained by simple calculations.

\begin{lemma}[Bourguignon-Lawson] \label{lemma:BL} Let $B=\iota_{\nabla f} \varphi $ for some $\varphi\in \Omega^2(\mathfrak{so}(E))$ with $(d_D)^* \varphi =0$, where $f$ is a smooth function on $M$. Then $(d_D)^* B=0$.
\end{lemma}

The proof of the theorem follows a parallel strategy.

\begin{proposition}\label{prop:33} Assume that $R^D$ is Yang-Mills of a Yang-Mills connection $D$ of $E$. Let $B_\ell =\iota_{\nabla \ell} R^D=R^D_{\nabla \ell, (\cdot)}$. Then
$$
\mathcal{J}_D (B_\ell) =-(m-4) B_\ell.
$$
Namely if $B_\ell\ne 0$, it is an eigenform of the Jacobi operator with eigenvalue $-(m-4)$.
\end{proposition}
\begin{proof} By calculations similar to that in the proof of Proposition \ref{prop:31} we have that
\begin{eqnarray*}
\Delta (\iota_{\nabla \ell} R^D)&=&  \Delta  (R^D_{\nabla \ell, (\cdot)})=\Delta (R_D)_{\nabla \ell, (\cdot)}+2\left(D_{e_j} R^D\right)_{\nabla_{e_j} \nabla \ell, (\cdot)}+R^D_{\Delta \nabla \ell, (\cdot)}\\
&=&\mathcal{R}^D (R^D)_{\nabla \ell, (\cdot)}+(2m-4)R^D_{\nabla \ell, (\cdot)}-R^D_{\nabla \ell, (\cdot)}\\
&=& \mathcal{R}^D (R^D)_{\nabla \ell, (\cdot)}+(2m-5)R^D_{\nabla \ell, (\cdot)}.
\end{eqnarray*}
From the line 1 to 2 we have used the Yang-Mills equation $(d_D)^* R^D=-\left(D_{e_j} R^D\right)_{e_j , (\cdot)}=0$ and $\nabla_{e_j} \nabla \ell=-\ell e_j$ to annihilate  $2\left(D_{e_j} R^D\right)_{\nabla_{e_j} \nabla \ell, (\cdot)}$. Here we also used (\ref{eq:31}), (\ref{eq:32}), Proposition \ref{prop:22}, precisely (\ref{eq:26}), and   that on $\mathbb{S}^m$, $2\Ric^M\wedge \operatorname{id}-2R^M=(2m-4)\operatorname{I}$, with $\operatorname{I}$ being the identity map  $\operatorname{I}:\mathfrak{so}(m)\to \mathfrak{so}(m)$. Summarizing the above we have that
\begin{equation}\label{eq:33}
\Delta B_\ell =\mathcal{R}^D (R^D)_{\nabla \ell,  (\cdot)}+(2m-5)B_\ell.
\end{equation}

Now apply Lemma \ref{lemma:BL} and Proposition \ref{prop:22} again, precisely (\ref{eq:25}),  and have that
\begin{eqnarray*}
\mathcal{J}_D (B_\ell)&=&\Delta_{d_D}(B_\ell)+\mathcal{R}^D(B_\ell)\\
&=&(m-1)B_\ell +2\mathcal{R}^D(B_\ell)-\mathcal{R}^D (R^D)_{\nabla \ell, (\cdot)}-(2m-5)B_\ell.
\end{eqnarray*}
In the second line above we also used (\ref{eq:33}). The claimed result follows if we can establish
\begin{equation}\label{eq:34}
2\mathcal{R}^D(B_\ell)-\mathcal{R}^D (R^D)_{\nabla \ell, (\cdot)}=0.
\end{equation}
Indeed for any $X$,
\begin{eqnarray*}
2\mathcal{R}^D(B_\ell)(X)-\mathcal{R}^D (R^D)_{\nabla \ell, (X)}&=& 2\sum_{j=1}^m [R^D_{e_j, X}, B_\ell(e_j)]\\
&\quad& -\sum_{j=1}^m \left( [R^D_{e_j, \nabla \ell}, R^D_{e_j, X}]-[R^D_{e_j, X}, R^D_{e_j, \nabla \ell}]\right)\\
&=&\sum_{j=1}^m [R^D_{e_j, X}, R^D_{\nabla \ell, e_j}]+[ R^D_{\nabla \ell, e_j}, R^D_{e_j, X}]=0.
\end{eqnarray*}
This completes the proof of the proposition.
\end{proof}

\begin{proposition}\label{prop:34} If for a linear function $\ell$, $B_\ell=\iota_{\nabla \ell} R^D\in \Omega^1(\mathfrak{so}(E))$ vanishes on $\mathbb{S}^m$, then $R^D$ is flat.
\end{proposition}
\begin{proof} This was proved in Proposition 4.3 of \cite{NU}. Here we  provide an alternate simple argument.
Let $\Psi_s(p)$ be the flow generated by $\nabla \ell$. Let $\widetilde{\Psi}_s$ be its lift on the related principle bundle. Let $D_s=\widetilde{\Psi}_{-s} D$ be a family of connections. It is known (cf. (2.34) of \cite{BL}) that
$$
B_\ell=\iota_{\nabla \ell} R^D=\left.\frac{d}{ds}\right|_{s=0} D_s.
$$
Here the definition of the lifting $\widetilde{\Psi}_s$ depends on $D$ via a horizontal lifting of $\nabla \ell$ requiring that $\Psi_s$ is identity map when $s=0$.  
Hence if for some $\ell$ $B_\ell=0$, it  implies that $D_s$ is  constant. This shows that $D$ is the connection of the  pull back from the fiber over one point, namely $p_{-\infty}$. This implies that $R^D$ is flat. One can also see this by defining the  $D$ and $D_s$ via the connexion (namely mappings, which smoothly depends on the paths satisfying some additional axioms, between the fibers over the two ends of smooth pathes) as on page 445 of \cite{Ni-holo}. The argument does not assume that  $R^D$ is a Yang-Mills fields.
\end{proof}

The part of Theorem \ref{thm:main} concerning the Yang-Mills fields now follows exactly as the previous case for harmonic maps.

\section{Minimal submanifolds in $\mathbb{S}^n$}

The argument of the previous section also implies a similar result for minimal submanifolds in $\mathbb{S}^n$. Let $N(M)$ denote the normal bundle of $M$. Here the key is the Codazzi equation, namely for $X, Y, Z $ tangent to $M$
\begin{equation}\label{eq:41}
(\nabla_X B)(Y, Z)-(\nabla_Y B)(X, Z)=0, \mbox{ where } B(X, Y)=\bar{\nabla}_XY-\nabla_X Y
\end{equation}
where $\bar{\nabla}$ being the Levi-Civita connection of $\mathbb{S}^n$ and $\nabla_XY$ being the induced connection of $M$ via the isometric immersion. Recall that $(\nabla_X B)(Y, Z)=\nabla^{\perp} (B(X, Y))-B(\nabla_XY, Z)-B(Y, \nabla_X Z)$ and $\nabla^{\perp}_X V=(\bar{\nabla}_X V)^\perp$ for any section $V$ of $N(M)$. Here $(\cdot)^{\perp}$ stands for the projection to the normal bundle.  To put into the setting of our previous discussion we define $\beta: T_pM \to T_pM \otimes N_p(M)$ as
$$
\langle \beta(X), Y\otimes V \rangle :=\langle B(X, Y), V\rangle:=\langle A^V(X), Y\rangle.
$$
The last equation defines $A^V: T_pM \to T_pM$, a the symmetric tensor of $T_pM$ (given any $p\in M$) for any $V\in N_p(M)$. The connection $D$ on $TM \otimes N(M)$ is defined via $\nabla$ and $\nabla^{\perp}$.
The $1$-form $\beta$ defined as above is a  $1$-form valued in $TM \otimes N(M)$. Direct calculation shows that the Codazzi equation is equivalent to $d_D \beta =0$.

The trace of the  second fundamental form $B: T_pM\times T_pM \to N_p(M)$, namely   $H:=\sum_{j=1}^mB(e_j, e_j)=0$,  for an orthonormal basis $\{e_i\}$ of $T_pM$, is called the mean curvature.  $M$ is a   minimal submanifold if and only if $H=0$. Below we show that $H=0$ implies that $\beta$ is a $d_\nabla$-harmonic $1$-form. In fact for an orthonormal basis $\{e_i\}$ of $T_pM$ with the property $\nabla_{e_i} e_j=0$ at a given $p$,
\begin{eqnarray}
\langle -(d_D)^*\beta, X\otimes V\rangle &=&\sum_{i=1}^m\langle(\nabla_{e_i} B)(e_i, X), V\rangle\nonumber\\
&=&\sum_{i=1}^m\nabla_{e_i}\langle \beta(e_i), X\otimes V\rangle-\langle \beta(\nabla_{e_i} e_i), X\otimes V\rangle \nonumber\\
&=& \sum_{i=1}^m\nabla_{e_i}\langle B(e_i, X), V\rangle-\langle B(\nabla_{e_i} e_i, X),  V\rangle \nonumber\\
&=&\sum_{i=1}^m\langle  (\nabla_{X} B)(e_i, e_i), V\rangle=0.\label{eq:42}
\end{eqnarray}
In the above we also used the Codazzi equation.
Summarizing the discussion we have that the  vanishing of the  mean curvature implies that $\beta$ satisfies  $(d_D)^*\beta =0$. Note that $d_D$ involves the induced connection $\nabla^{\perp}$ (which defined as $(\bar{\nabla}_X V)^{\perp}$) on $N(M)$. In fact the argument above also proves the following proposition.

\begin{proposition}\label{prop:41}
Let $M$ be an isometric immersed submanifold in $\mathbb{S}^n$ (or in any space forms with constant sectional curvature). Then $M$ has parallel mean curvature $H$ if and only if $1$-form  $\beta$ defined above is  $d_D$-harmonic.
\end{proposition}

This gives a characterization of submanifolds with parallel mean curvature, similar to that of Ruh-Vilms \cite{R-V}, which is formulated in terms of the harmonicity of  the Gauss map into the corresponding Grassmanian manifolds. The proposition above seems  easier to use.
The second variation of the area $\mathcal{A}$ for a minimal submanifold $M$ (inside another Riemannian manifold $N$) has the following form (\cite{Simons-Ann}, Theorem 3.2.2):
\begin{equation}\label{eq:43}
\left.\frac{d^2}{ds^2}\mathcal{A}(M_s)\right|_{s=0}=\int_M \langle \nabla^*\nabla V-\bar{R}(V)-\tilde{A}(V), V\rangle=:\int_M \langle \mathcal{J}_M (V), V\rangle,
\end{equation}
where $V$ is the variational vector field, which in this case belongs to $\Gamma(N(M))$. Here $\nabla^*\nabla =-\Delta$ of $N(M)$ (namely with respect to $\nabla^{\perp}$).
 For two  sections of the normal bundle $V, W$, the following defines the operators $\bar{R}$ and $\tilde{A}$, which are symmetric transformations of $N_p(M)$:
$$
\bar{R}(V)=\sum_{j=1}^m (R^{N}_{V, e_j}e_j)^{\perp}, \quad \quad \langle \tilde{A}(V), W\rangle =\sum_{j=1}^m \langle (\bar{\nabla}_{e_j} V)^{T}, (\bar{\nabla}_{e_j} W)^{T}\rangle.
$$
Here $(\cdot)^T$ stands the projection to $T_pM$ at any given point $p\in M$. In the case that $N=\mathbb{S}^{n}$, we have that  $\bar{R}(V)=mV$. Motivated by Proposition \ref{prop:41} and the results of last section we have the following result for minimal submanifolds.

\begin{theorem}\label{thm:main2} Let $M^m$ be a minimal submanifold of $\mathbb{S}^n$. Let $N(M)$ denote the normal bundle of $M$ inside $\mathbb{S}^n$. For the associated Jacobi operator $\mathcal{J}_M$, let $E^M_\lambda=\{ V\in \Gamma(N)\, |\, \mathcal{J}_M (V)=\lambda V\}$ be the eigenspace.
Then $\dim(E^M_{-m})\ge n-m$. The equality holds if and only if $M$ is isometric to $\mathbb{S}^m$. In particular $\lambda_1(\mathcal{J}_M)\le -m$.
\end{theorem}
\begin{proof} Let $V_\ell$ be the projection of $\bar{\nabla} \ell$ to the normal bundle $N(M)$. It is a section of $N(M)$. Denote the tangential projection of $\bar{\nabla} \ell$ to $TM$ by $T_\ell$. We calculate $\mathcal{J}_M(V_\ell)$.  Pick an orthogonal frame $\{e_j\}_{j=1}^m$ with $\nabla_{e_i} e_j =0$ at a fixed point $p\in M$. Then for a section  $W$ of $N(M)$,
\begin{eqnarray*}
\langle \mathcal{J}_M(V_\ell), W\rangle &=& -\langle \sum_{j=1}^m \bar{\nabla}_{e_j} \left(\bar{\nabla}_{e_j} V_\ell\right)^{\perp}, W\rangle-m\langle V_\ell, W\rangle- \sum_{j=1}^m \langle(\bar{\nabla}_{e_j} V_\ell)^{T}, (\bar{\nabla}_{e_j} W)^{T}\rangle;
\end{eqnarray*}
\begin{eqnarray*}
-\langle \sum_{j=1}^m \bar{\nabla}_{e_j} \left(\bar{\nabla}_{e_j} V_\ell\right)^{\perp}, W\rangle &=& -\langle \sum_{j=1}^m \bar{\nabla}_{e_j} \left(\bar{\nabla}_{e_j}(\bar{\nabla} \ell-T_\ell)\right)^{\perp}, W\rangle\\
&=& \langle  \sum_{j=1}^m \bar{\nabla}_{e_j} \left(B(e_j, T_\ell)\right), W\rangle \\
&=& \langle  \sum_{j=1}^m B(e_j, \nabla_{e_j} T_\ell), W\rangle.
\end{eqnarray*}
From line 2 to 3 we have used (\ref{eq:31}) and from line 3 to 4 we have used that the Codazzi equation and $\nabla^{\perp}_X ( \sum_{j=1}^m B(e_j, e_j))=0$. Finally
\begin{eqnarray*}
\langle  \sum_{j=1}^m B(e_j, \nabla_{e_j} T_\ell), W\rangle &=&  \sum_{j=1}^m \langle A^W(e_j), \nabla_{e_j} T_\ell\rangle= \sum_{j=1}^m \langle A^W(e_j), \bar{\nabla}_{e_j} T_\ell\rangle \\
&=& \sum_{j=1}^m \langle A^W(e_j), \bar{\nabla}_{e_j} (\bar{\nabla}\ell-V_\ell)\rangle\\
&=& -\langle W, \ell \sum_{j=1}^mB(e_j, e_j)\rangle  +\sum_{j=1}^m \langle (\bar{\nabla}_{e_j} W)^{T}, (\bar{\nabla}_{e_j} V_\ell)^{T}\rangle\\
&=&\sum_{j=1}^m \langle (\bar{\nabla}_{e_j} W)^{T}, (\bar{\nabla}_{e_j} V_\ell)^{T}\rangle.
\end{eqnarray*}
Here we have used that $ A^W(e_j)=-(\bar{\nabla}_{e_j} W)^{T}$ and $\sum_{j=1}^mB(e_j, e_j)=0$. Putting the above together we have that
$$
\langle \mathcal{J}_M(V_\ell), W\rangle =-m\langle V_\ell, W\rangle
$$
for any local section $W$ of $N(M)$. This proves that $V_\ell\in E^M_{-m}$. Now let $\mathcal{S}=\{V_\ell\, |\, \ell \in \mathcal{H}^1\}$. Since for any $p\in M\subset \mathbb{S}^n$, $\dim(\{\nabla \ell (p)\,|\, \ell \in \mathcal{H}^1\})=n$. Hence  $\dim(\mathcal{S}_p)=n-m$ with $\mathcal{S}_p=\{V_\ell(p)|\, V_\ell\in \mathcal{S}\}$. This proves the lower multiplicity estimate since $\dim(\mathcal{S})\ge \dim(\mathcal{S}_p)$.

It is easy to see that the standard embedding of $\mathbb{S}^m\to \mathbb{S}^n$ attains  the equality of the theorem. Let $\mathcal{G}:=\{\nabla \ell\,|\, \ell \in \mathcal{H}^1\}$. As before $\dim(\mathcal{G})=n+1.$ Let $r_p: \mathcal{G}\to T_p \mathbb{S}^n$ be the restriction map. Let $\mathcal{N}$ and $\mathcal{T}$ be the projections from $\mathcal{G}$ to $N(M)$ and $T(M)$. Then $\mathcal{S}=\mathcal{N}(\mathcal{G})$. If the equality holds we have that $\dim(\mathcal{S})=n-m$. This implies that
$\dim(\ker(\mathcal{N}))=n+1-(n-m)$.  Clearly $\ker{\mathcal{N}}\subset \ker (\mathcal{N}_p\cdot r_p).$ Since $\dim(\ker (\mathcal{N}_p\cdot r_p))=n+1-(n-m)$, we have that $\ker{\mathcal{N}}= \ker (\mathcal{N}_p\cdot r_p)$ for any $p\in M$. Namely if $(\bar{\nabla} \ell)^{\perp}(p)=0$, $(\bar{\nabla} \ell)^{\perp}\equiv 0$ on $M$. This implies that
$$
0=\nabla_X V_\ell(=\nabla^{\perp}_X V_\ell)=(\bar{\nabla}_X (\bar{\nabla} \ell-T_\ell))^{\perp}=\left((-\ell X)-B(X, T_\ell)\right)^{\perp}=-B(X, (\bar{\nabla}\ell)^T).
$$
Since such $(\bar{\nabla}\ell)^T$ spans $T_pM$ as $\ell$ varies, the above equation implies that $M$ is totally geodesic.
\end{proof}

We remark that the Morse index estimate of minimal submanifolds in $\mathbb{S}^n$ was proved in \cite{Simons-Ann}, and the above  proof for the equality case is similar to the corresponding rigidity for the Morse index of \cite{Simons-Ann}. Since that  the equality holds for the multiplicity lower bound estimate does not necessarily imply that the Morse index (which can be strictly larger than the multiplicity) lower estimate holds the equality, the result above does not follow from Simons' result. When $M$ is a hypersurface, which is not a totally geodesic sphere, an eigenvalue  estimate   for $\mathcal{J}_M$  was obtained in  \cite{Simons-Ann} (cf. Lemma 6.1.7 there). It played a crucial role in Simons' proof of Bernstein conjecture by applying it to the analysis on the cone over $M$. The above result provides a high codimensional analogue.  It is expected that the eigenvalue estimate of Theorem \ref{thm:main2} plays some role in the study of high codimensional minimal surfaces. The Morse index of fully immersed (non-stable) $\mathbb{S}^2$ in $\mathbb{S}^{2n}$ was calculated in \cite{Ejiri} to be $2(n(n+2)-3)$.

\section*{Acknowledgments} {} We would like to thank Jiaping Wang for his interest, Zhuhuan Yu for the reference \cite{NU}.

\end{document}